\documentclass[10pt,leqno]{article}
\usepackage{amsmath}
\usepackage{amsfonts}
\usepackage{amsthm}
\usepackage{indentfirst}

\theoremstyle{plain}
\newtheorem{theorem}{\indent\rm T\,h\,e\,o\,r\,e\,m\;}[section]

\newtheorem{proposition}{\indent\rm P\,r\,o\,p\,o\,s\,i\,t\,i\,o\,n\;}[section]
\newtheorem{corollary}{\indent\rm C\,o\,r\,o\,l\,l\,a\,r\,y\;}[section]

\theoremstyle{definition}
\newtheorem{definition}{\indent\rm D\,e\,f\,i\,n\,i\,t\,i\,o\,n\;}[section]

\theoremstyle{remark}
\newtheorem{remark}{\indent\rm R\,e\,m\,a\,r\,k\;}[section]
\newtheorem{notations}{\indent N\,o\,t\,a\,t\,i\,o\,n\,s\;}[section]

\renewenvironment{proof}{\indent\rm P\,r\,o\,o\,f.\;}{\hfill $\square$ \\ \indent}

\def\ints{\int\limits}

\makeatletter                                                  %
\renewcommand*{\@seccntformat}[1]{
  \csname the#1\endcsname\;-                                   %
}                                                              %
\renewcommand{\section}{\@startsection{section}{1}{0mm}        %
   {1.5\baselineskip}
   {1\baselineskip}
   {\indent\normalfont\normalsize\bfseries}
   }                                                           %
\renewcommand*{\@seccntformat}[1]{
  \normalfont\bfseries\csname the#1\endcsname\;-               %
}                                                              %
\renewcommand\subsection{\@startsection                        %
  {subsection}{2}{0mm}
  {1.5\baselineskip}
  {1\baselineskip}
  {\indent\normalfont\normalsize\itshape}}
\renewcommand*{\@seccntformat}[1]{
  \normalfont\bfseries\csname the#1\endcsname\;-               %
}                                                              %
\renewcommand\subsubsection{\@startsection                     %
  {subsubsection}{2}{0mm}
  {1.5\baselineskip}
  {1\baselineskip}
  {\indent\normalfont\normalsize\texttt}}
\makeatother                                                   %
\begin{document}
\thispagestyle{empty}

\parindent 0pt


\begin{center}
{\sc\large Alberto Dolcetti}
\ {\small and}
 \ {\sc\large Donato Pertici}
\end{center}
\vspace {1.5cm}

\centerline{\large{\textbf{ Some differential properties of $GL_n(\mathbb{R})$}}}
\smallskip
\centerline{\large{\textbf{ with the trace metric}}}

\renewcommand{\thefootnote}{\fnsymbol{footnote}}

\footnotetext{
This research was partially supported by MIUR-PRIN: ``Variet\`a reali e complesse: geo\-me\-tria, topologia e analisi armonica''}.

\renewcommand{\thefootnote}{\arabic{footnote}}
\setcounter{footnote}{0}

\vspace{1,5cm}
\begin{center}
\begin{minipage}[t]{10cm}

\small{ \noindent \textbf{Abstract.} In this note we consider some properties of $GL_n(\mathbb{R})$ with the Semi-Riemannian structure induced by the trace metric $g$. In particular we study geodesics and curvature tensors. Moreover we prove that $GL_n$ has a suitable foliation, whose leaves are isometric to $(SL_n(\mathbb{R}), g)$, while its component of matrices with positive determinant is isometric to the  Semi-Riemannian product manifold $SL_n \times \mathbb{R}$.
\medskip

\noindent \textbf{Keywords.} Trace metric, Semi-Riemannian manifold, geodesic, curvature tensors, nonsingular (special) matrix, exponential and logarithm of a real matrix.
\medskip

\noindent \textbf{Mathematics~Subject~Classification~(2010):}
53C50, 53C22, 15A16.

}
\end{minipage}
\end{center}

\bigskip

\section*{Introduction}

The so-called trace metric $g$: $g_A(V,W) = tr(A^{-1}VA^{-1}W)$ for any $A \in GL_n$ and any $V, W \in T_A(GL_n)$ ($tr$ indicates the trace of a matrix) induces a Semi-Riemannian structure on $GL_n=GL_n(\mathbb{R})$.  The metric $g$ is often studied in the context of positive definite real matrices on which it defines a structure of Riemannian manifold. The geometry of the Riemannian manifold of positive definite real matrices has recently been object of interest in different frameworks. We refer the reader for instance to \cite{Lan1999} Ch.XII, \cite{LawL2001}, \cite{ALM2004}, \cite{BhaH2006} \S2, \cite{Bha2007} Ch.6, \cite{MoZ2011} for more details and further information on this subject. In particular geodesic arcs between two positive definite matrices have been studied in details, because their middle point is their expected geometric mean. Of course existence, uniqueness and explicit descriptions of geodesics have a fundamental role in this setting and are reached in many ways: for instance as consequences of an exponential increasing metric property (see for instance \cite{BhaH2006}, \cite{Bha2007} Ch.6) or as solutions of the second-order differential equation $\stackrel{..}{P} - \stackrel{.}{P} \stackrel{ }P^{-1} \stackrel{.}{P} = 0$ with certain initial data (see \cite{MoZ2011} Sec.3.5).

In this note we generalize some arguments used in the second approach and prove that $(GL_n, g)$ is Semi-Riemannian with signature
$(\dfrac{n(n+1)}{2}, \dfrac{n(n-1)}{2})$ (Proposition \ref{homog}), whose connected components, $(GL_n^+, g)$ and $(GL_n^-, g)$, are symmetric manifolds (Proposition \ref{isomGL}).

The characterization of geodesics of $(GL_n, g)$ is in Theorem \ref{Levi-Civita_geod_GL}, where also the Levi-Civita connection is described. As in the case of positive definite matrices of \cite{MoZ2011}, the geodesics are solutions of the previous differential equation. Moreover geodesic arcs between two points $K_0$ and $K_1$ of $GL_n$ correspond to real solutions of the exponential matricial equation $exp(X) = K_0^{-1} K_1$ (Corollary \ref{Cor_to_Levi-Civita_geod_GL}) and so we are able to translate the existence of these geodesic arcs in Theorem \ref{Culver_for_GL_n} in terms of Jordan form of $K_0^{-1}K_1$ by means of \cite{Culv1966}. In particular any two points of $GL_n$ can be always joined by a geodesic arc or by a singly broken geodesic arc (Proposition \ref{singly_brocken}). When the geodesic arc is unique, we give its explicit expression (Proposition \ref{expression_unique_geodesic_arc}). By the way we observe that the Levi-Civita connection of $(GL_n, g)$ is the Cartan-Schouten $(0)$-connection of $GL_n$ (Corollary \ref{rem_before_0_conn}).

Afterwards we compute the Riemann curvature tensors of type $(1,3)$ and of type $(0,4)$ and the sectional curvature (Proposition \ref{RiemannTens_GL}) and also the Ricci curvature and the scalar curvature of $(GL_n, g)$ (Proposition \ref{Ricci_GL}).

Finally we focus our attention on $SL_n(\mathbb{R}) = SL_n$, where the metric $g$ sets up a structure of   Einstein, symmetric, totally geodesic, Semi-Riemannian submanifold of $(GL_n, g)$ (Propositions \ref{symmetric_semiriem_SL},  \ref{geodSL} and \ref{SL_Einstein}) and we show that $GL_n$ has a foliation, whose leaves are Einstein, symmetric, geodesically complete, totally geodesic, Semi-Riemannian submanifolds and isometric to $(SL_n, g)$ (Theorem \ref{foliation_GL_SL}) and furthermore we show that its component of matrices with positive determinant is isometric to the Semi-Riemannian product manifold $(SL_n \times \mathbb{R}, g \times h)$ where $h$ is the euclidean metric on $\mathbb{R}$ (Theorem \ref{GL+_prod}).

Similar situations seem to appear even in case of some particular submanifolds of $(GL_n,g)$. These are the subject of further works currently in preparation.

We refer to \cite{O'N1983} for all standard facts on Semi-Riemannian manifolds and in particular for notions and notations not explicitally recalled here and also to \cite{Helg2001} for Riemannian symmetric spaces, while we refer to \cite{Hi2008} and to \cite{HoJ1985} for standard facts about matrices and exponential function.

\section{The Semi-Riemannian manifold $GL_n(\mathbb{R})$}

\begin{definition}
A Semi-Riemannian manifold $(M, g)$ is a smooth real manifold $M$ endowed with a metric tensor $g$, i.e. a symmetric nondegenerate $(0,2)$ tensor field $g$ of constant signature.
\end{definition}

\begin{remark}\label{signature_semir_homog}
If $g$ is simply supposed to be a symmetric $C^\infty$-tensor of type $(0,2)$ on $M$ and $M$ is supposed to be homogeneous, i.e. for every $p_1,p_2\in M$ there is a diffeomorphism $F: M \to M$ with $F(p_1)=p_2$ and preserving $g$, then $g$ is nondegenerate (i.e. it is a Semi-Riemannian metric) if and only if it is so at one point.
\end{remark}

\begin{notations}
We denote by $M_n= M_n(\mathbb{R})$ and $GL_n = GL_n(\mathbb{R})$ respectively the vector space of real square matrices of order $n$ and the multiplicative group of nondegenerate matrices in $M_n$.  $GL_n$ is a Lie group of dimension $n^2$ with two connected components, depending on their determinant: $GL_n^+$ and $GL_n^-$. $M_n$ is the Lie algebra of $GL_n$ and the tangent space of $GL_n$ at $A \in GL_n$ is $T_A(GL_n) = M_n$. $SL_n$ is the connected Lie subgroup of $GL_n$ of matrices with determinant $1$ and we put $SL_n(c) = \{ M \in GL_n \ / \ \det (M) = c \}$ for any $c \in \mathbb{R} \setminus \{0\}$.

$S_n$ and $A_n$ are the vector subspaces of $M_n$ of symmetric and skew symmetric matrices respectively.

As usual $I=I_n$ is the identity matrix, $[A,B] = AB-BA$ for any $A, B \in M_n$ and also $[X,Y] = X \circ Y - Y\circ X$ for any $X$, $Y$ vector fields on $GL_n$. 

We define a $C^\infty$-tensor $g$ of type $(0,2)$ on $GL_n$, by $g_A(V,W) = tr(A^{-1}VA^{-1}W)$ ($tr$ indicates the trace of a matrix). This tensor induces a metric, called also trace metric, often considered in the context of positive definite real matrices on which it defines the structure of Riemannian manifold (see for instance \cite{Lan1999} Ch.XII, \cite{BhaH2006} \S2, \cite{Bha2007} Ch.6, \cite{MoZ2011} \S3).

From now on, $g$ will indicate this tensor.
\end{notations}

\begin{proposition}\label{homog}
$(GL_n, g)$ is a homogeneous Semi-Riemannian manifold with signature 
$(\dfrac{n(n+1)}{2}, \dfrac{n(n-1)}{2})$.
\end{proposition}

\begin{proof}
Let us consider the left translation $L_G: GL_n \to GL_n, X \mapsto GX$ and the right translation $R_G: GL_n \to GL_n, X \mapsto XG$ and prove that both preserve the tensor $g$ for every $G \in GL_n$.

Indeed $L_G$ and $R_G$ are both linear, hence $(DL_G)_A = L_G$ and $(DR_G)_A = R_G$ at each point $A \in GL_n$. Therefore for any $A \in GL_n$ and any $V, W \in M_n$ we have: 

\begin{multline*}
g_{L_G(A)} ((DL_G)_A(V), (DL_G)_A(W)) = g_{GA}(GV,GW)= \\
tr((GA)^{-1}GV(GA)^{-1}GW)= tr(A^{-1}VA^{-1}W) = g_A(V,W)
\end{multline*}

and analogously: 

$$g_{R_G(A)} ((DR_G)_A(V), (DR_G)_A(W)) = g_A(V,W).$$

The invariance of $g$ under left and right translations implies that both translations are isometries. This allows to deduce that $(GL_n, g)$ is a homogeneous manifold: if $A, B$ are in $GL_n$, then, for instance, the left translation $L_{BA^{-1}}$ preserves the tensor $g$ and maps $A$ to $B$. Hence to conclude, by Remark \ref{signature_semir_homog}, it is sufficient to argue for the single point $I=I_n$.

First we note that $g_I: M_n \times M_n \to \mathbb{R}$ is obviously a symmmetric bilinear form.

Now let $V$ be a matrix such that $g_I(V,W) = tr(VW)=0$ for every matrix $W \in M_n$. For $W=V^T$ (the transpose of $V$) we get $tr(VV^T)=0$, this suffices to get $V=0$, so $g_I$ is nondegenerate.

For any $S \in S_n$ and any $A \in A_n$ we have $g_I(S,A)=0$. Indeed $g_I(S,A)=tr(SA) = tr((SA)^T) = tr(A^TS^T) = - tr(AS) = - tr(SA) = -g_I(S,A)$. Moreover it is easy to check that $g_I(S,S) \ge 0$ with equality if and only if $S=0$ and that $g_I(A,A) \le 0$ with equality if and only if $A=0$. This gives that the restriction $g_I|_{S_n \times S_n}$ is positive definite, that the restriction $g_I|_{A_n \times A_n}$ is negative definite and that $S_n$ and $A_n$ are orthogonal with respect to $g_I$. Now $M_n = S_n \oplus A_n$, hence it follows that the index of positivity of $g_I$ is $\dfrac{n(n+1)}{2}$ (the dimension of $S_n$) and its index of negativity is $\dfrac{n(n-1)}{2}$ (the dimension of $A_n$), so the signature of $g_I$ is $(\dfrac{n(n+1)}{2}, \dfrac{n(n-1)}{2})$.
\end{proof}

\begin{remark}
As both left and right translations are isometries of $(GL_n,g)$, so all their compositions are; in particular: the opposite $X \mapsto -X=L_{-I_n}(X) = R_{-I_n}(X)$, the conjugacies $C_G$: $C_G(X)= G^{-1}XG$ and the congruences $\Gamma_G$: $\Gamma_G(X)= G^T XG$ ($G \in GL_n$).

Now let us denote by $\varphi: GL_n \to GL_n$ the inversion map, i.e. $\varphi(A) = A^{-1}$; $\varphi$ is a diffeomorphism of $GL_n$ onto itself with $\varphi^2= Id_{GL_n}$ and differential $(D\varphi)_A(V) = -A^{-1}VA^{-1}$ for any $A \in GL_n$ and any $V \in M_n$. 
Therefore we have:
\begin{multline*}
g_{\varphi(A)}((D\varphi)_A(V), (D\varphi)_A(W))= 
g_{A^{-1}}(-A^{-1}VA^{-1}, -A^{-1}WA^{-1}) = \\
 tr(A(-A^{-1}VA^{-1})A(-A^{-1}WA^{-1})) = tr(VA^{-1}WA^{-1}) = g_A(V,W) 
\end{multline*}
Hence also $\varphi$ is an isometry of $(GL_n, g)$.

Let us denote by $\tau: GL_n \to GL_n$ the transposition, i.e. $\tau(A) = A^T$; also $\tau$ is a diffeormorphism of $GL_n$ onto itself with $\tau^2= Id_{GL_n}$ and it is an isometry. Indeed its differential is $\tau$ itself being linear and, after denoting by $A^{-T} = (A^T)^{-1} = (A^{-1})^T$, we have 
\begin{multline*}
g_{\tau (A)}((D\tau)_A(V), (D\tau)_A(W)) =  \\
tr(A^{-T}V^TA^{-T}W^T) = tr(WA^{-1}VA^{-1}) = g_A(V,W).
\end{multline*}

Note that the symmetric nondegenerate matrices are the fixed points of the isometry $\tau$ on $GL_n$. 

Finally we recall that a Semi-Riemannian (globally) symmetric space is a connected Semi-Riemannian manifold $M$ such that for each $p \in M$ there is a (unique) isometry $\zeta_p: M \to M$ with differential map $-id$ on $T_p M$ and fixing $p$.
\end{remark}

\begin{proposition}\label{isomGL}
1) Among the isometries of the Semi-Riemannian manifold $(GL_n,g)$ there are the left translations $L_G$, and the right translations $R_G$, the conjugacies $C_G$, the congruences $\Gamma_G$ ($G \in GL_n$), the opposite, the inversion $\varphi$, the transpose $\tau$ and all their compositions.

2) Both $(GL_n^+, g)$ and $(GL_n^-, g)$ are symmetric manifolds and for any $A \in GL_n^+$ (or in $GL_n^-$)  the symmetry with respect to $A$ is $\psi_A = R_A \circ L_A \circ \varphi = L_A \circ R_A \circ \varphi$. In particular $\psi_{I_n} = \varphi$.
\end{proposition}
\begin{proof}
Part (1) has been proved in the previous remark. For (2): for every $A \in GL_n^+$ (or in $GL_n^-$), $\psi_A$ is an isometry of $(M_n,g)$. We have $\psi_A(X) = AX^{-1}A$, therefore $\psi_A(A)=A$ and $(D\psi_A)_A=R_A \circ L_A \circ D\varphi_A$, hence $(D\psi_A)_A(W)= (R_A \circ L_A)(-A^{-1}WA^{-1})= -W$, so $(D\psi_A)_A = - id_{T_A(GL_n)}$.
\end{proof}

\section{Geodesics in $(GL_n, g)$}

\begin{notations}\label{base_euclidean_derivative}
Let $P=(p^{ij}) \in GL_n$, where $p^{ij}$ indicates the the $(i,j)$-entry of $P$. We denote by $\{E_{ij}\}$, $1 \le i,j \le n$, the standard basis of $M_n$, where $E_{ij} \in M_n$ is the matrix whose entries are $0$ except for the $(i,j)$-entry which is $1$. After reordering, $\{E_{ij}\}$ can be rewritten as $\{E_\alpha\}$, $1 \le \alpha \le n^2$, just following the columns one after another. Hence we can write $P= \sum_{\alpha}p^{\alpha}E_{\alpha}$ with $p^{\alpha} \in \mathbb{R}$. The $p^{\alpha}$, $1 \le \alpha \le n^2$ are natural coordinates on the whole $GL_n$ and $(p^1, \cdots , p^{n^2})$ runs over an open subset of $\mathbb{R}^{n^2}$. $M_n$ is the tangent space to $GL_n$ at each point, hence we can identify $E_\alpha$ with $\dfrac{\partial \ }{\partial p^\alpha}$ for any $\alpha = 1, \cdots , n^2$.

Now if $X = \sum_{\alpha = 1}^{n^2} X^\alpha E_\alpha$, $Y = \sum_{\alpha = 1}^{n^2} Y^\alpha E_\alpha$ are tangent vector fields of class $C^\infty$ on $GL_n$, we can define a new tangent vector field of class $C^\infty$ on $GL_n$: the euclidean derivative of the field $Y$ along the field $X$ and indicated by $X(Y)$, by setting $X(Y) = \sum_{\alpha, \beta = 1}^{n^2} X^\alpha \dfrac{\partial Y^\beta}{\partial p^\alpha} E_\beta$.
\end{notations}

\begin{theorem}\label{Levi-Civita_geod_GL}
1) Let $\nabla$ be the Levi-Civita connection of $(GL_n,g)$. If $X$ and $Y$ are tangent vector fields of class $C^\infty$ on $GL_n$, then 
$$(\nabla_X Y)_{_P} = (X(Y))_{_P} -\dfrac{1}{2} (X_{_P}P^{-1}Y_{_P} + Y_{_P}P^{-1}X_{_P})$$ 
for any $P \in GL_n$, where $X(Y)$ is the euclidean derivative of $Y$ with respect to $X$.

2) Let $P=P(t)$ be a $C^\infty$-curve on $(GL_n, g)$, then $P$ is a geodesic if and only if 
$$\nabla_{\stackrel{.}{P}} \stackrel{.}{P} \ = \  \stackrel{..}{P} - \stackrel{.}{P} \stackrel{ }P^{-1} \stackrel{.}{P} \ = 0$$
where $\stackrel{.}{P}$ and $\stackrel{..}{P}$ are the first and the second derivative of $P$ with respect to $t$.

3) The geodesics of $(GL_n, g)$ are precisely the curves of the type: 
$$P(t)=Ke^{tC}$$ 
for any $C \in M_n$ and any $K \in GL_n$.

4) $(GL_n, g)$ is a geodesically complete Semi-Riemannian manifold.
\end{theorem}

\begin{proof}
In this proof we generalize the arguments developed by \cite{MoZ2011} \S 3, in case of positive definite  matrices.

We indicate  by $g_{\alpha \beta} = g_{\alpha \beta}(p^1, \cdots , p^{n^2}) = g_P(E_\alpha, E_\beta)$ the components of the metric tensor $g$ with respect to the euclidean coordinates $(p^1, \cdots , p^{n^2})$ and by $g^{\alpha \beta} = g^{\alpha \beta}(p^1, \cdots , p^{n^2})$ the entries of the inverse of the matrix $(g_{\alpha \beta})$ which is invertible at any point, because the metric $g$ is nondegenerate on $GL_n$. Hence we have: $\sum_{\beta =1}^{n^2}g_{\alpha \beta}g^{\beta \gamma} = \delta _\alpha^\gamma$ (Kronecker symbol).

Now let $\nabla$ be the Levi-Civita connection associated to the Semi-Riemannian metric $g$. To simplify the notations we omit the index $P$.
We have: $\nabla_{E_\alpha} E_\beta = \sum_{\gamma = 1}^{n^2} \Gamma_{\alpha \beta}^{\gamma}E_\gamma$, where the Christoffel symbols $\Gamma_{\alpha \beta}^{\gamma}$ can be expressed as $\Gamma_{\alpha \beta}^{\gamma} = \sum_{\delta =1}^{n^2} \dfrac{g^{\gamma \delta}}{2} (g_{\alpha \delta, \beta} + g_{\beta \delta, \alpha} - g_{\alpha \beta , \delta})$ with $g_{\alpha \beta , \delta} = \dfrac{\partial g_{\alpha \beta}}{\partial p^{\delta}}$ for any $\alpha, \beta, \delta \in \{ 1, \cdots , n^2 \}$.

\medskip

\emph{Claim 1.}
For any $\alpha, \beta, \gamma \in \{ 1, \cdots , n^2 \}$ we have: 

$$\Gamma_{\alpha \beta}^\gamma = - \dfrac{1}{2} \sum_{\delta=1}^{n^2} g^{\gamma \delta} \{tr(P^{-1}E_\alpha P^{-1}E_\beta P^{-1}E_\delta) + tr(P^{-1}E_\beta P^{-1}E_\alpha P^{-1}E_\delta)\}.$$

\medskip

Indeed, remembering that $\dfrac{\partial}{\partial p^{\delta}}(P^{-1}) = - P^{-1} E_\delta P^{-1}$, standard computations show that 
\begin{multline*}
g_{\alpha \beta , \delta} = \dfrac{\partial \ }{\partial p^{\delta}} (tr(P^{-1}E_\alpha P^{-1}E_\beta)) = \\
- tr(P^{-1} E_\alpha P^{-1} E_\beta P^{-1} E_\delta) - tr(P^{-1} E_\beta P^{-1} E_\alpha P^{-1} E_\delta).
\end{multline*}

Hence 
\begin{multline*}
\Gamma_{\alpha \beta}^\gamma = \sum_{\delta =1}^{n^2} \dfrac{g^{\gamma \delta}}{2} \{- tr(P^{-1} E_\alpha P^{-1} E_\delta P^{-1} E_\beta) - tr(P^{-1} E_\delta P^{-1} E_\alpha P^{-1} E_\beta) \\
\shoveright{-tr(P^{-1} E_\delta P^{-1} E_\beta P^{-1} E_\alpha) - tr(P^{-1} E_\beta\ P^{-1} E_\delta P^{-1} E_\alpha) } \\
\shoveright{+tr(P^{-1} E_\alpha P^{-1} E_\beta P^{-1} E_\delta) + tr(P^{-1} E_\beta P^{-1} E_\alpha P^{-1} E_\delta)\} =} \\
 - \dfrac{1}{2} \sum_{\delta =1}^{n^2} g^{\gamma \delta} \{tr(P^{-1} E_\alpha P^{-1} E_\beta P^{-1} E_\delta) +tr(P^{-1} E_\beta P^{-1} E_\alpha P^{-1} E_\delta)\}
\end{multline*} 

as predicted.

\medskip

An elementary computation of linear algebra allows us to get also 

\medskip

\emph{Claim 2.} Let $V= \sum_{\alpha = 1}^{n^2} V^\alpha E_\alpha$ be a vector field on $GL_n$. 

Then for any $\alpha = 1, \cdots , n^2$ we have $V^\alpha = \sum_{\beta = 1}^{n^2} g^{\alpha \beta} g(V, E_\beta)$ and so 

$V = \sum_{\alpha, \beta = 1}^{n^2} g^{\alpha \beta} g(V, E_\beta) E_\alpha$.

\medskip

Now let $X = \sum_{\alpha = 1}^{n^2} X^\alpha E_\alpha$, $Y = \sum_{\beta = 1}^{n^2} Y^\beta E_\beta$ be as in (1). Hence:

\begin{multline*}
\nabla_X Y = \sum_{\alpha, \beta = 1}^{n^2} X^\alpha \nabla_{E_\alpha}(Y^\beta E_\beta) = \\
\sum_{\alpha, \beta = 1}^{n^2} X^\alpha \dfrac{\partial Y^\beta}{\partial p^\alpha} E_\beta + \sum_{\alpha, \beta = 1}^{n^2} X^\alpha Y^\beta \nabla_{E_\alpha} E_\beta = X(Y) + \sum_{\alpha, \beta, \gamma=1}^{n^2} X^\alpha Y^\beta \Gamma_{\alpha \beta}^\gamma E_\gamma
\end{multline*}

which by Claim 1 is equal to
\begin{multline*}
X(Y)- \dfrac{1}{2} \sum_{\alpha, \beta, \gamma, \delta = 1}^{n^2} X^\alpha Y^\beta g^{\gamma \delta} \{tr(P^{-1} E_\alpha P^{-1} E_\beta P^{-1} E_\delta) +\\
 \shoveright{tr(P^{-1}E_\beta P^{-1} E_\alpha P^{-1} E_\delta)\}E_\gamma =}\\
X(Y) - \dfrac{1}{2} \sum_{\gamma, \delta = 1}^{n^2} g^{\gamma \delta} \{tr (P^{-1} X P^{-1} Y P^{-1} E_\delta) + tr(P^{-1} Y P^{-1}X P^{-1}E_\delta)\} E_\gamma =\\
X(Y) - \dfrac{1}{2} \sum_{\gamma, \delta = 1}^{n^2} g^{\gamma \delta} \{tr(P^{-1}(XP^{-1}Y + YP^{-1}X)P^{-1}E_\delta)\} E_\gamma =\\
X(Y) -\dfrac{1}{2} \sum_{\gamma, \delta = 1}^{n^2} g^{\gamma \delta} g_{_P}(XP^{-1}Y + YP^{-1}X, E	_\delta) E_\gamma .
\end{multline*}
This, by Claim 2, is $X(Y) - \dfrac{1}{2} (XP^{-1}Y + YP^{-1}X)$ and we conclude (1).

\smallskip

Now (2) follows from (1), because the euclidean derivative of $\stackrel{.}{P}$ with respect to $\stackrel{.}{P}$ is $\stackrel{..}{P}$.

\smallskip

From $\stackrel{..}{P} - \stackrel{.}{P} \stackrel{ }P^{-1} \stackrel{.}{P} \ = 0$ we get $P^{-1}\stackrel{..}{P} - P^{-1}\stackrel{.}{P} \stackrel{ }P^{-1} \stackrel{.}{P} \ = 0$, hence, remembering that $\dfrac{d}{dt}(P^{-1}) = - P^{-1} \stackrel{.}{P} P^{-1}$, we get $\dfrac{d}{dt}(P^{-1}\stackrel{.}{P}) = 0$,  so $P^{-1}\stackrel{.}{P} = C$, constant. Then $\stackrel{.}{P} = PC$, so $\stackrel{.}{P} e^{-tC} - PC \, e^{-tC} = 0$. For any constant matrix $X$ we have $\dfrac{d}{dt}(e^{tX}) = X e^{tX} = e^{tX} X$, so we deduce that $\dfrac{d}{dt}(Pe^{-tC}) = 0$, hence $P(t) e^{-tC} = K$, constant with $det(K) \ne 0$ and in conclusion $P(t)=Ke^{tC}$, as predicted in (3). Finally we get (4), because any maximal geodesic is clearly defined on the entire real line.
\end{proof}

\begin{remark}\label{rem_before_0_conn}
As in Notations \ref{base_euclidean_derivative}, let $X= \sum_{\alpha = 1}^{n^2} X^\alpha(p) \dfrac{\partial \ }{\partial p^\alpha}$ be a $C^\infty$-vector field on $GL_n$ and let us denote by $P= \sum_{\alpha = 1}^{n^2}p^\alpha E_\alpha$ (it can be also viewed as  a $C^\infty$-vector field on $GL_n$), then we have: $X(P) = X$.

Let us denote by $\mathcal{GL}_n$ the Lie algebra of $GL_n$.

\smallskip

Let $X_0 \in T_{I_n}(GL_n) = M_n$. The unique left-invariant vector field $X \in \mathcal{GL}_n$, assuming the value $X_0$ at the identity, is the field $X$ defined by $X_P = P X_0$ for any $P \in GL_n$. Then we get: $[X,Y]_{_P}= P(X_0Y_0 -Y_0X_0)$ for any $P \in GL_n$, where $X,Y \in \mathcal{GL}_n$ are such that $X_{I_n} = X_0$, $Y_{I_n} = Y_0$.

\smallskip

If $X,Y \in \mathcal{GL}_n$ are such that $X_{I_n} = X_0$, $Y_{I_n} = Y_0$, then $(X(Y))_{_P} = PX_0Y_0$ for any $P \in GL_n$.

Indeed, taking into account the previous facts, we have: 
$
(X(Y))_{_P} = PX_0(P)Y_0 = P X_0 Y_0
$.

Now by these facts and by Theorem \ref{Levi-Civita_geod_GL}, if $\nabla$ is the Levi-Civita connection of $(GL_n, g)$, we can get: $(\nabla_XY)_{_P} = (X(Y))_{_P} - \dfrac{1}{2} (X_P P^{-1}Y_P + Y_P P^{-1}X_P) = \dfrac{1}{2} [X,Y]_{_P}$ for any $X, Y \in \mathcal{GL}_n$ and any $P \in GL_n$. This allows to state the following
\end{remark}

\begin{corollary}\label{Cartan_0_Conn}
Let $\nabla$ be the Levi-Civita connection of $(GL_n, g)$. Then
$$
\nabla_X Y = \dfrac{1}{2} [X,Y] \in \mathcal{GL}_n
$$
for any $X, Y \in \mathcal{GL}_n$.

Hence $\nabla$ is the Cartan-Schouten $(0)$-connection of $GL_n$ (see \cite{Helg2001} p.148 and pp.549-550).
\end{corollary}

\begin{corollary}\label{Cor_to_Levi-Civita_geod_GL}
1) The curve $P(t)=Ke^{tC}$, $C \in M_n$ and $K \in GL_n$, is the unique geodesic of $(GL_n, g)$ emaning from $K$ with velocity $KC$ at $t=0$ and vice versa the unique geodesic of $(GL_n, g)$ emaning from $K \in GL_n$ with velocity $S \in M_n$ at $t=0$ is: 
$P(t)= K  exp(tK^{-1} S)$.

2) Let $K_0, K_1 \in GL_n$; a geodesic arc joining  $K_0$ and $K_1$ in  $GL_n$ is any geodesic $\gamma: [0,1] \to GL_n$ such that $\gamma(0) = K_0$, $\gamma(1) = K_1$. Then there exists a geodesic arc in $(GL_n, g)$ joining $K_0, K_1$ if and only if the exponential equation $exp(X)= K_0^{-1}K_1$ has a real solution $C$, moreover the real solutions correspond bijectively to the geodesics of $(GL_n, g)$ starting from $K_0$ at $t=0$ and passing through $K_1$ at $t=1$.

3) If $K$ is a positive definite symmetric real matrix and $K^{\frac{1}{2}}$ denotes its unique positive definite square root matrix and if $S \in S_n$, then the unique geodesic emaning from $K$ with velocity $S$ at $t=0$ is: $P(t)=K^{\frac{1}{2}} exp(t K^{-\frac{1}{2}}S K^{-\frac{1}{2}}) K^{\frac{1}{2}}$ (see for instance \cite{MoZ2011} thm.3.5).
\end{corollary}

\begin{proof}
The first part of (1) follows by remarking that $P(0)=K$ and $\stackrel{.}{P}(0)= KC$ is the velocity at $t=0$.

If $K \in GL_n$ and $S \in M_n$ the unique geodesic emaning from $K$ with velocity $S$ (for existence and uniqueness, remember for instance \cite{O'N1983} p.68 lemma 22) can be only the above curve: this completes (1).

To prove (2), assume that such a geodesic arc, $P(t)$, exists. By Theorem \ref{Levi-Civita_geod_GL}, $P(t)= K_0 exp(tC)$ for some $C \in M_n$ and so $K_1 = P(1)= K_0 exp(C)$. Hence we can conclude that $exp(C) = K_0^{-1} K_1$. 
For the converse suppose that $C$ is a real matrix with $exp(C)= K_0^{-1}K_1$. By part (3) of Theorem \ref{Levi-Civita_geod_GL}, the curve $P(t)=K_0 exp(tC)$ (the unique geodesic emaning from $K_0$ with velocity $K_0C$ at $t=0$) passes through $K_1$ too, because $P(1) = K_0 exp(C) = K_0 K_0^{-1}K_1 = K_1$. 
We conclude that distinct solutions $C, C'$ of the previous exponential equation correspond to distinct geodesic arcs with prescribed endpoints: indeed the correspondig geodesic arcs have in $K_0$ velocities $K_0C$ and $K_0C'$ which must be distinct, otherwise $C=C'$.

Finally by means of standard properties of $exp$ we can write: 
\begin{multline*}
P(t) = K exp(tK^{-1} S) = \\
K^{\frac{1}{2}} K^{\frac{1}{2}} exp(tK^{-\frac{1}{2}} K^{-\frac{1}{2}}S K^{-\frac{1}{2}} K^{\frac{1}{2}}) = K^{\frac{1}{2}} exp(t K^{-\frac{1}{2}}S K^{-\frac{1}{2}}) K^{\frac{1}{2}}.
\end{multline*}
\end{proof}

\begin{theorem}\label{Culver_for_GL_n}
1) Let $K_0, K_1 \in GL_n$. Then there exists a geodesic arc of $(GL_n, g)$ joining $K_0$, $K_1$ if and only if 
each elementary divisor (Jordan block) of $K_0^{-1}K_1$ belonging to any (possible) negative eigenvalue occurs an even number of times; moreover the geodesic arc is unique if and only if all the eigenvalues of $K_0^{-1}K_1$ are positive real and no elementary divisor (Jordan block) of $K_0^{-1}K_1$ belonging to any eigenvalue appears more than once.

2) Assume that there is more of one geodesic arc of $(GL_n, g)$ joining $K_0$, $K_1$. Then there exists an infinity of such geodesic arcs, which are

(a) countable if $K_0^{-1}K_1$ has complex eigenvalues none of which belongs to more than one Jordan block and all (possible) real eigenvalues of $K_0^{-1}K_1$ are positive such that their Jordan blocks appear only once;

(b) uncountable (more precisely a countinuous) if $K_0^{-1}K_1$  has some negative real eigenvalue, or if it has some positive real eigenvalues belonging to Jordan blocks that appear more than once, or it has some complex conjugate eigenvalues belonging to more than one Jordan block.
\end{theorem}

\begin{proof}
The point (2) of Corollary \ref{Cor_to_Levi-Civita_geod_GL} translates the existence of geodesic arcs in $(GL_n, g)$, joining $K_0$ and $K_1$, into the existence of real solutions of the esponential equation $exp(X)= K_0^{-1}K_1$. The study of the equation $exp(X) = M$, $M \in M_n$ has been accomplished by W. J. Culver in \cite{Culv1966} and it depends on Jordan form of $K_0^{-1} K_1$.
So (1) translates the existence of a real solution of the previous exponential equation (\cite{Culv1966} thm.1) and characterizes its uniqueness (\cite{Culv1966} thm.2), while (2) describes the cases of its nonuniqueness (\cite{Culv1966} cor.).
\end{proof}

\begin{remark}
The condition in (1) of Theorem \ref{Culver_for_GL_n} implies $\det(K_0^{-1}K_1) > 0$. The positivity of this determinant is equivalent to say that $K_0$, $K_1$ belong both to $GL_n^+$ or to $GL_n^-$, which is of course obvious for the existence of a geodesic arc between them. Then the point (1) of the previous theorem points out that this fact is only necessary, but not sufficient, for the existence of a geodesic arc between $K_0$ and $K_1$.

When $K_0 = I_n$, then $K_0^{-1}K_1 = K_1$. Hence Jordan form of $K_0^{-1}K_1$ is nothing but Jordan form of $K_1$. Note that we can always reduce to this case, because there are some isometries (for instance the left translation $L_{K_0^{-1}}$) mapping $K_0$ to $I_n$.

The next corollaries follow directly from Theorem \ref{Culver_for_GL_n}.
\end{remark}

\begin{corollary}\label{Culver_for_diag}
Let $K_0, K_1 \in GL_n$ and assume that $K_0^{-1}K_1$ is similar to a diagonal real matrix $diag(\lambda_1, \cdots , \lambda_n)$.

There exists a geodesic arc in $(GL_n, g)$ joining $K_0$ and $K_1$ if and only if any (possible) negative $\lambda_i$ appears an even number of times.

There is a unique geodesic arc in $(GL_n, g)$ joining $K_0$ and $K_1$ if and only if $\lambda_1, \cdots , \lambda_n$ are positive and distinct.

Assume that there is more than one geodesic arc in $(GL_n, g)$ joining $K_0$ and $K_1$, then there exists a continuous of such geodesic arcs and there is a negative $\lambda_i$ (which appears an even number of times) or there is a positive $\lambda_j$ which appears more than one time.
\end{corollary}

\begin{corollary}\label{example_n=2-3}
Let $K_0$, $K_1$ be matrices both either in $GL_n^+$ or in $GL_n^-$ (so $K_0^{-1}K_1 \in GL_n^+$).

\smallskip

Case $n=2$. 

There exists a unique geodesic arc joining $K_0$, $K_1$ if and only if all eigenvalues of $K_0^{-1}K_1$ are real positive and $K_0$, $K_1$ are linearly independent matrices.

There are countably many geodesic arcs joining $K_0$, $K_1$ if and only if the eigenvalues of $K_0^{-1}K_1$ are not real.

There is an uncountable family of geodesic arcs joining $K_0$, $K_1$ if and only if they are linearly dependent matrices.

In any other case there is no geodesic arc joining $K_0$, $K_1$.

\smallskip

Case $n=3$. 

There exists a unique geodesic arc joining $K_0$, $K_1$ if and only if either all eigenvalues of $K_0^{-1}K_1$ are real positive and distinct or they are real positive and $K_0^{-1}K_1$ is not diagonalizable.

There are countably many geodesic arcs joining $K_0$, $K_1$ if and only if $K_0^{-1}K_1$ has a positive eigenvalue and the others are not real.

There is an uncountable family of geodesic arcs joining $K_0$, $K_1$ if and only if $K_0^{-1}K_1$ is diagonalizable over $\mathbb{R}$ and at least two eigenvalues are equal.

In any other case there is no geodesic arc joining $K_0$, $K_1$.
\end{corollary}

\begin{remark}\label{log_Jordan}
Assume now that $K_0, K_1 \in GL_n$ and that there is a unique geodesic arc joining them (remember Theorem \ref{Culver_for_GL_n}). We want to write down explicitally this geodesic arc.

Let $J_k(\lambda) = \lambda I_k + N_k$ be the Jordan block of order $k$ and eigenvalue $\lambda$ with $N_k$ the upper-triangular matrix whose entry $(i,j)$ is $\delta_{i+1,j}$. Standard computations (for instance on formal series of matrices) show that, if $\lambda \in \mathbb{R}$, $\lambda > 0$, the unique real logarithm matrix of $J_k(\lambda)$ (i.e. the unique real solution of $exp(Y) = J_k(\lambda)$) is 

$$
Y = LOG(J_k(\lambda)) = (log \lambda) I_k + \sum_{i=1}^{k-1} \dfrac{(-1)^{i+1}}{i \ \lambda^i} N_k^i
$$

where $log \lambda$ is the real natural logarithm of $\lambda$.

For any $t \in \mathbb{R}$ we have 

$$
exp(t LOG(J_k(\lambda))) = \lambda^t (I_k + \sum_{s=1}^{k-1} \binom{t}{s} \dfrac{N_k^s}{\lambda^s})
$$

where 
$$\binom{t}{s} = \dfrac{t(t-1) \cdots (t-s+1)}{s!}$$ 
and we set 
$$
J_k(\lambda)^t =(\lambda I_k + N_k)^t = \lambda^t (I_k + \sum_{s=1}^{k-1} \binom{t}{s} \dfrac{N_k^s}{\lambda^s}).
$$

Now let $X \in GL_n$ be a matrix such that $X= C^{-1} diag(J_{k_1}(\lambda_1), \cdots , J_{k_p}(\lambda_p))C$, with $C \in GL_n$, $\lambda_1, \cdots , \lambda_p > 0$ and $(\lambda_i, k_i) \ne (\lambda_j, k_j)$ as soon as $i \ne j$. Then the unique real logarithm of $X$ can be written as 

$$
LOG(X) = C^{-1} diag(LOG (J_{k_1}(\lambda_1)), \cdots , LOG (J_{k_p}(\lambda_p)))C.
$$

For any $t \in \mathbb{R}$ we pose $X^t = exp(t LOG(X))$ and so we get

$$
X^t = C^{-1} diag(J_{k_1}(\lambda_1)^t, \cdots , J_{k_p}(\lambda_p)^t)C.
$$

Taking into account \ref{Cor_to_Levi-Civita_geod_GL} we can state the following
\end{remark}

\begin{proposition}\label{expression_unique_geodesic_arc}
If there is a unique geodesic arc joining $K_0, K_1 \in GL_n$, then it can be written as $\gamma (t) = K_0 (K_0^{-1}K_1)^t$.
\end{proposition}

\begin{corollary}
If there is a unique geodesic arc joining $K_0, K_1 \in GL_n$ then any two distinct points on the geodesic, to which this arc belongs, are joined by a unique geodesic arc (and of course the geodesic, to which this new arc belongs, overlaps to the previous one).
\end{corollary}

\begin{proof}
From the previous proposition the unique geodesic arc joining $K_0, K_1 \in GL_n$ is $\gamma (t) = K_0 (K_0^{-1}K_1)^t$. Now let $P, Q$ be on the corresponding geodesic, i.e. $P= K_0 (K_0^{-1}K_1)^r$, $Q= K_0 (K_0^{-1}K_1)^s$ for some $r,s \in \mathbb{R}$, $ r\ne s$. This gives $P^{-1}Q =  (K_0^{-1}K_1)^{-r} (K_0^{-1}K_1)^s= (K_0^{-1}K_1)^{s-r}$. 

Now $K_0^{-1}K_1 = C^{-1} diag(J_{k_1}(\lambda_1), \cdots , J_{k_p}(\lambda_p))C$, with $C \in GL_n$,  
$\lambda_1, \cdots , \lambda_p > 0$ and $(\lambda_i, k_i) \ne (\lambda_j, k_j)$ as soon as $i \ne j$ (remember Remark \ref{log_Jordan}), so we get $P^{-1}Q = C^{-1} diag(J_{k_1}(\lambda_1)^{s-r}, \cdots , J_{k_p}(\lambda_p)^{s-r})C$,  whose Jordan form is the Jordan form of $K_0^{-1}K_1$ with eigenvalues $\lambda_j^{s-r}$ instead of $\lambda_j$ (remember that $r \ne s$). We can conclude with Theorem \ref{Culver_for_GL_n}.
\end{proof}

\begin{proposition}\label{singly_brocken}
Let $K_1$, $K_2$ be matrices both in $GL_n^+$ (resp. $GL_n^-$). Then $K_1$, $K_2$ can always be joined by a singly broken geodesic arc in $(GL_n^+, g)$ (resp. $(GL_n^-, g)$).
\end{proposition}
\begin{proof}
We prove that for any $K_1, K_2$ as above there is a nonsingular matrix $Z$ that can be joined by a geodesic arc with both $K_1, K_2$. This fact, together with Theorem \ref{Culver_for_GL_n}, allows to conclude.

For any $K_1, K_2 \in GL_n^+$ we can consider their polar decompositions $K_1 = O_1 P_1$, $K_2= P_2 O_2$ with $O_1, O_2$ special orthogonal real matrices and $P_1, P_2$ positive definite real matrices. We denote $Z= P_2 O_1 = O_1 \overline{P}$ with $\overline{P}= O_1^T P_2 O_1$: note that also $\overline{P}$ is positive definite. We have $K_1^{-1} Z = P_1^{-1} O_1^T O_1 \overline{P} = P_1^{-1} \overline{P}$. Now $P_1^{-1}$ and $\overline{P}$ are simultaneously diagonalizable under congruences and $P_1^{-1} \overline{P}$ is similar to a nonsingular diagonal real matrix with positive eigenvalues, hence by Corollary \ref{Culver_for_diag} there is a geodesic arc in $(GL_n, g)$, joining $Z$ and $K_1$.

On the other hand $K_2^{-1} Z = O_2^T P_2^{-1} P_2 O_1 = O_2^T O_1$ which is in $SO_n$. The elements of $SO_n$ are similar to diagonal complex matrices in which, if a negative real eigenvalue appears, it is $-1$ and appears an even number of times, thus by Theorem \ref{Culver_for_GL_n} there is a geodesic arc in $(GL_n, g)$, joining $Z$ and $K_2$.

Analogous arguments work, if $K_1, K_2 \in GL_n^-$. Indeed now the polar decompositions are $K_1 = O_1 P_1$, $K_2= P_2 O_2$ with $O_1, O_2$ orthogonal real matrices with negative determinant and $P_1, P_2$ positive definite real matrices. Again $\overline{P}$ is positive definite and $K_2^{-1} Z = O_2^T O_1 \in SO_n$.
\end{proof}

\section{Curvature of $(GL_n, g)$}

\begin{proposition}\label{RiemannTens_GL}
Let $K \in GL_n$ and $X, Y, Z \in M_n$.

1) The Riemann curvature tensor of type $(1,3)$ of $GL_n$ at $K$ is
$$
(R_{XY}Z)_{_K} =  - \dfrac{1}{4} ( Z \, [K^{-1}X, K^{-1}Y] - [XK^{-1}, YK^{-1}]\, Z ).
$$

2) The Riemann curvature tensor of type $(0,4)$ of $GL_n$ at $K$ is
$$R_{XYZW}(K) = \dfrac{1}{4} tr([K^{-1}X, K^{-1}Y] \, [K^{-1}Z, K^{-1}W]).$$

3)  If $s_K$ is a nondegenerate $2$-section of $T_K(GL_n)$ (i.e. a $2$-dimensional subspace of $T_K(GL_n)= M_n$ such that the restriction of $g_K$ to $s_K \times s_K$ is a nondegenerate symmetric bilinear form) and $X, Y$ are linearly independent vectors in $s_K$, then the sectional curvature of $(GL_n, g)$ on $s_K$ is 
$$
\mathcal{K}(s_K) = \dfrac{1}{4} \dfrac{tr([K^{-1}X,K^{-1}Y]^2)}{g_{_K}(X,X)g_{_K}(Y,Y) - g_{_K}(X,Y)^2}.$$
\end{proposition}

\begin{proof}
If $\{E_\alpha\}$, $1 \le \alpha \le n^2$ is the basis contructed in Notations \ref{base_euclidean_derivative}, for $X, Y, Z \in T_K(GL_n) = M_n$, we have $X= \sum_{\alpha = 1}^{n^2} X^{\alpha} \dfrac{\partial}{\partial p^\alpha}\vert_{_K}$, $Y= \sum_{\beta = 1}^{n^2} Y^{\beta} \dfrac{\partial}{\partial p^\beta}\vert_{_K}$, $Z= \sum_{\gamma = 1}^{n^2} Z^{\gamma} \dfrac{\partial}{\partial p^\gamma}\vert_{_K}$, where $X_\alpha, Y_\beta, Z_\gamma \in \mathbb{R}$. 

We can extend in a natural way $X, Y, Z$ to $C^\infty$-vector fields with constant coefficients on $GL_n$, we still call $X, Y, Z$. Then for any $Q \in GL_n$ we have $X_Q= \sum_{\alpha = 1}^{n^2} X^{\alpha} \dfrac{\partial}{\partial p^\alpha}\vert_{_Q}$, $Y_Q= \sum_{\beta = 1}^{n^2} Y^{\beta} \dfrac{\partial}{\partial p^\beta}\vert_{_Q}$, $Z_Q= \sum_{\gamma = 1}^{n^2} X^{\gamma} \dfrac{\partial}{\partial p^\gamma}\vert_{_Q}$ and so 
$X(Z)$, $Y(Z)$, $X(Y)$, $Y(X)$ are identically zero. 

Hence, by Theorem \ref{Levi-Civita_geod_GL}, $(\nabla_YZ)_{_Q} = -\dfrac{1}{2} (Y Q^{-1} Z + ZQ^{-1}Y)$. Again, by Theorem \ref{Levi-Civita_geod_GL}, we get
\begin{multline*}
(\nabla_X(\nabla_Y Z))_{_K} = (X(\nabla_Y Z))_{_K} -\dfrac{1}{2} \{ X K^{-1} (\nabla_Y Z)_K +  (\nabla_Y Z)_K K^{-1} X \} = \\
  \shoveleft{\dfrac{1}{2} (YK^{-1}XK^{-1}Z + Z K^{-1} X K^{-1} Y) +} \\
  \dfrac{1}{4} (X K^{-1} Y K^{-1} Z + X K^{-1} Z K^{-1} Y + Y K^{-1} Z K^{-1} X + Z K^{-1} Y K^{-1} X).
\end{multline*}

Interchanging $X$ and $Y$ we get another analogous formula.

$X, Y$ are vector fields with constant coefficients with respect to the coordinate fields $E_\alpha$, so, by Schwarz rule, we have $[X,Y] = 0$; therefore at $K$ we get: 
\begin{multline*}
(R_{XY}Z)_{_K} = - (\nabla_X(\nabla_YZ))_{_K} + (\nabla_Y(\nabla_XZ))_{_K} = \\ 
- \dfrac{1}{4} (ZK^{-1}(XK^{-1}Y - Y K^{-1} X) + (Y K^{-1} X - X K^{-1} Y) K^{-1}Z) =\\
 - \dfrac{1}{4} ( Z \, [K^{-1}X, K^{-1}Y] - [XK^{-1}, YK^{-1}]\, Z ).
\end{multline*}
This completes (1).

We get (2) by standard computations remembering (1):
\begin{multline*}
R_{XYZW}(K) = g_K(R_{XY}Z,W) = \\
  \dfrac{1}{4} (tr\{(K^{-1}XK^{-1}Y- K^{-1}YK^{-1}X)(K^{-1}ZK^{-1}W-K^{-1}WK^{-1}Z)\})=\\
  \dfrac{1}{4} tr([K^{-1}X, K^{-1}Y] \, [K^{-1}Z, K^{-1}W]).
\end{multline*}

Finally we get (3) from (2), because $\mathcal{K}(s_K) = \dfrac{R_{XYXY}}{g_{_K}(X,X)g_{_K}(Y,Y)-g_{_K}(X,Y)^2}$ and it does not depend on the generators $X, Y$.
\end{proof}

\begin{corollary}\label{R_{XY}Z}
With the same notations as in Remark \ref{rem_before_0_conn}, if $R$ is the Riemann curvature tensor of type $(1,3)$ of $(GL_n, g)$, then 
$$
R_{XY}Z = \dfrac{1}{4} [[X,Y],Z] \in \mathcal{GL}_n
$$
for any $X,Y,Z \in \mathcal{GL}_n$, which can be written in the form 
$$
R_{XY} = \dfrac{1}{4} ad([X,Y])
$$
for any for any $X,Y \in \mathcal{GL}_n$ (see for instance \cite{Helg2001} p.99-100).
\end{corollary}
\begin{proof}
By Proposition \ref{RiemannTens_GL} at a point $K \in GL_n$ and with the notations of Remark \ref{rem_before_0_conn}, we have: 

$
R_{X_KY_K}Z_K =  \dfrac{1}{4} ( -Z_K \, [K^{-1}X_K, K^{-1}Y_K] + [X_KK^{-1}, Y_KK^{-1}]\, Z_K )
$, which by standard computations becomes $\dfrac{K}{4}[[X_0,Y_0],Z_0]$. Hence, always by Remark \ref{rem_before_0_conn}, the former is $\dfrac{1}{4}[[X,Y], Z]_{_K}$, which allows to conclude.
\end{proof}

\begin{remark}\label{base_orton}
Let $E_{ij}$, $1 \le i, j \le n$ be the matrix whose entries are zero everywhere except for the entry $(i,j)$ which is $1$. Easy computations show that at the point $I=I_n \in GL_n$ an orthonormal basis for $g_I$ is 
\begin{multline*}
\{D_i= E_{ii} / i=1, \cdots , n \} \ \cup \\
\{ S_{ij}= \dfrac{E_{ij}+E_{ji}}{\sqrt{2}} / 1 \le i < j \le n \} \ \cup \\
\{ A_{ij}= \dfrac{E_{ij}-E_{ji}}{\sqrt{2}} / 1 \le i < j \le n \}.
\end{multline*}

The vectors $D_i$'s and $S_{ij}$'s are space-like, while the vectors $A_{ij}$'s are time-like.
\end{remark}

\begin{proposition}\label{Ricci_GL}
Let $Ric_K$ be the Ricci curvature tensor of $(GL_n, g)$ at $K \in GL_n$, then for any $X, Y \in T_K(GL_n) = M_n$ we have 
$$Ric_K(X,Y) = \dfrac{1}{2} tr(K^{-1}X) tr(K^{-1}Y) - \dfrac{n}{2} g_K(X,Y).$$

Moreover $(GL_n, g)$ is a Semi-Riemannian manifold whose scalar curvature is constant and equal to $S= - \dfrac{(n+1)n(n-1)}{2}$.
\end{proposition}

\begin{proof} 
The formula on $Ric(X,Y)$ can be obtained by standard but long (and tedious) computations. Next we give shorter (and perhaps more elegant) arguments involving the Cartan-Killing form (for standard facts on it see for instance \cite{Helg2001} p.131 and \cite{FulHar1991}).

In general $Ric(X,Y)$ is the trace of the map $Z \mapsto R_{XZ}Y$, where we can suppose $X, Y, Z$ left invariant. By Corollary \ref{R_{XY}Z} this trace is the trace of $Z \mapsto \dfrac{1}{4} [[X,Z], Y] = - \dfrac{1}{4}[Y,[X,Z]] = -\dfrac{1}{4} (ad_Y \circ ad_X) (Z) = - \dfrac{1}{4} B(Y,X) = - \dfrac{1}{4} B(X,Y)$ where $B$ is the Cartan-Killing form of $\mathcal{GL}_n$. 

Now it is known that $B(X,Y) = 2n \ tr(XY) -2 \ tr(X)tr(Y) = 2n \ g(X,Y) - 2\ tr(X)tr(Y)$ (see for instance \cite{FulHar1991} p.210). 

Then $Ric_I(X,Y) = \dfrac{1}{2}tr(X)tr(Y) -\dfrac{n}{2} g_I(X,Y)$ for any $X, Y \in \mathcal{GL}_n$.

More generally for any $X,Y \in T_K(GL_n)$ we get the expected formula for $Ric_K(X,Y)$.

Finally let $S$ be the scalar curvature of $(GL_n, g)$, which is homogeneous, so $S$ is constant, because it is invariant under isometries. Hence it suffices to compute it at the point $I=I_n$. By Remark \ref{base_orton}: 

$S= \sum_{i=1}^n Ric_I(D_i,D_i) + \sum_{1 \le i < j \leq n} Ric_I(S_{ij}, S_{ij}) - \sum_{1 \le i<j \le n} Ric_I(A_{ij}, A_{ij})$. 

But the first part of this proposition gives 

$Ric_I(D_i,D_i) = - \dfrac{(n-1)}{2}$ 
for any $i= 1, \cdots  n$, 

$Ric_I(S_{ij}, S_{ij}) = \dfrac{1}{2} (tr(S_{ij}))^2 - \dfrac{n}{2} g_I(S_{ij}, S_{ij}) = -\dfrac{n}{2}$ for any $1 \le i < j \le n$, 

$ Ric_I(A_{ij}, A_{ij}) = \dfrac{1}{2} (tr(A_{ij}))^2 - \dfrac{n}{2} g_I(A_{ij}, A_{ij})=  \dfrac{n}{2}$ for any $1 \le i < j \le n$. 

Putting together the previous computations, we easily conclude the last statement too. 
\end{proof}

\section{The Semi-Riemannian manifold $SL_n(\mathbb{R})$}

\begin{remark}\label{tangent_SL}
For any $K \in SL_n $ we have: $T_K(SL_n) = \{ W \in M_n / tr(K^{-1} W)=0 \}$. 

\smallskip

Recall the Jacobi's formula: if $A = A(t)$  is a $C^1$-curve of $GL_n$ with $t \in (a,b) \subset \mathbb{R}$, then $\dfrac{d}{dt}(\det  A(t)) = \det  (A(t)) tr(A^{-1}(t) {\stackrel{.}{A\ }}\!(t))$ for any $t \in (a,b)$ where ${\stackrel{.}{A\ }} = \dfrac{dA}{dt}$. Then if $P=P(t)$ is a $C^\infty$-curve in $SL_n$ (hence $\det P(t) =1$ for any $t$), we get: $tr(P^{-1}(t) {\stackrel{.}{P}}(t)) = 0 $ for any $t$. This allows to conclude.

\smallskip

At the point $I=I_n \in SL \subset GL_n$ the identity matrix $I$ is a space-like vector, because $g_I(I,I) = tr(I) = n > 0$, whose perpendicular space is $Span(I)^\bot = \{ W \in M_n \  / \ g_I(I,W) = tr(W) = 0 \} = T_I(SL_n)$. Hence $M_n = Span(I) \oplus Span(I)^\bot = Span(I) \oplus T_I(SL_n)$. Given a point $P \in SL_n$, we denote again with $g_P$ the restriction to $T_P(SL_n) \times T_P(SL_n)$ of the tensor $g_P$ defined on $T_P(GL_n) \times T_P(GL_n)$. 
\end{remark}

\begin{proposition}\label{symmetric_semiriem_SL}
$(SL_n, g)$ is a symmetric Semi-Riemannian submanifold of $(GL_n, g)$ with signature $(\dfrac{n(n+1)}{2} -1, \dfrac{n(n-1)}{2})$. It is homogeneous and among its isometries there are left and right translations $L_G, R_G$, congruences $\Gamma_G$ ($G \in SL_n$), conjugacies $C_G$ ($G \in GL_n$) the transposition, the inversion $\varphi$ and all their compositions, so in particular the symmetries $\psi_P = R_P \circ L_P \circ \varphi$.
\end{proposition}

\begin{proof}
The metric $g_I$ is nondegenerate with signature $(\dfrac{n(n+1)}{2} -1, \dfrac{n(n-1)}{2})$ on $T_I(SL_n)$ . Indeed let $W \in T_I(SL_n)$ such that $g_I(V,W)=0$ for every $V \in T_I(SL_n)$. If $Z$ is any vector in $T_I(GL_n)= M_n$, then there exists a unique pair $(Z_0, \lambda_0) \in T_I(SL_n) \times \mathbb{R}$, such tha $Z = Z_0 + \lambda_0 I$. Hence $g_I(Z, W) = g_I(Z_0, W) + \lambda_0 g_I(I,W) = g_I(Z_0, W) + \lambda_0 tr(W) = g_I(Z_0, W)=0$, because $Z_0 \in T_I(SL_n)$. Now $g_I$ is nondegenerate on $T_I(GL_n)$, hence $W=0$, therefore $g_I$ is nondegenerate on $T_I(SL_n)$ too. Moreover,  $M_n = Span(I) \oplus T_I(SL_n)$ and $I$ is a space-like vector in $T_I(GL_n)=M_n$, so we get that the index of positivity of $g_I$ on $T_I(SL_n)$ is equal to the analogous index on $T_I(GL_n)$ minus $1$, hence the signature is $(\dfrac{n(n+1)}{2} -1, \dfrac{n(n-1)}{2})$. 
Now, for $P,Q \in SL_n$, the left translation $L_{QP^{-1}}$ is an isometry of $(GL_n, g)$ mapping $P$ into $Q$. Now the restriction of $L_{QP^{-1}}$ to $SL_n$ maps $SL_n$ into itself, so this restriction (denoted again by $L_{QP^{-1}}$) is an isometry of $(SL_n, g)$, which is therefore homogeneous. We conclude the analogous results on $(GL_n, g)$ proved above.
\end{proof}

\begin{proposition}\label{geodSL}
$(SL_n, g)$ is a totally geodesic Semi-Riemannian submanifold of $(GL_n, g)$.

The geodesics of $(SL_n, g)$ are precisely the curves of the type: 
$$P(t)=Ke^{tC}$$ with $\det  (K) = 1$ and  $tr(C) =0$ and $(SL_n, g)$ is geodesically complete.
\end{proposition}

\begin{proof}
As usual let $\nabla$ be the Levi-Civita connection of $(GL_n, g)$ and let $X,Y$ be vector fields, which are tangent to the submanifold $SL_n$. So, by Remark \ref{tangent_SL}, for any $P \in SL_n$: $tr(P^{-1}X_P)=tr(P^{-1}Y_P)=0$. The first part of the proposition follows from the fact that $(\nabla_XY)_P \in T_P(SL_n)$, i.e. again by Remark \ref{tangent_SL} from the fact that $tr(P^{-1} (\nabla_XY)_P)=0$ for any $P \in SL_n$. 

By Theorem \ref{Levi-Civita_geod_GL} we have: $(\nabla_XY)_P = (X(Y))_P - \dfrac {1}{2} \{ X_PP^{-1} Y_P + Y_P P^{-1}X_P \}$. Hence 
\begin{multline*}
P^{-1}(\nabla_XY)_P =\\
 P^{-1} (X(Y))_P - \dfrac {1}{2} P^{-1} X_PP^{-1} Y_P - \dfrac{1}{2} P^{-1} Y_P P^{-1}X_P = \\
 P^{-1} (X(Y))_P - P^{-1} X_PP^{-1} Y_P + \dfrac{1}{2} (P^{-1} X_PP^{-1} Y_P -P^{-1} Y_PP^{-1} X_P).
\end{multline*}

Now we have $tr(P^{-1}Y_P) =0$ for any $P \in SL_n$, so: 
\begin{multline*}
0 = X_P(tr(P^{-1} Y_P)) = tr(X_P(P^{-1}Y_P)) = \\
- tr( P^{-1} X_P P^{-1} Y_P) + tr( P^{-1} (X(Y))_{_P}).
\end{multline*}
Hence: 

$tr(P^{-1}(\nabla_XY)_P) = X_P(tr(P^{-1}Y_P)) + \dfrac{1}{2} tr(P^{-1}X_PP^{-1}Y_P- P^{-1} Y_P P^{-1}X_P) = \dfrac{1}{2} \{tr((P^{-1} X_P)(P^{-1}Y_P)) -tr((P^{-1}Y_P)(P^{-1}X_P))\} =0$ for any $P \in SL_n$ and so $(SL_n, g)$ is a totally geodesic submanifold.

Let us denote again by $\nabla$ the Levi-Civita connection of $(SL_n, g)$, from the first part of this proposition we get that the expression of $\nabla$ is formally similar to the expression of the Levi-Civita connection on $GL_n$, hence the same holds for the Riemann tensors and, analogously to Theorem \ref{Levi-Civita_geod_GL}, the equation of geodesics in $SL_n$ is the same and the geodesics are the curves of the predicted type.
\end{proof}

\begin{proposition}\label{SL_Einstein}
$(SL_n, g)$ is an Einstein Semi-Riemannian manifold with Ricci curvature tensor $Ric= - \dfrac{n}{2}g$ and scalar curvature $S= - \dfrac{(n-1)n(n+1)}{2}$.
\end{proposition}

\begin{proof}
Arguing as in the proof of Proposition \ref{Ricci_GL} we get 

$Ric(X,Y) = - \dfrac{1}{4} B(X,Y)$ for any left invariant fields $X,Y$ where $B$ is the Cartan-Killing form on the Lie algebra of $SL_n$ and that this is equal to $2n\, g(X,Y)$ (see for instance \cite{FulHar1991} p. 210) and so we get the expected formula for $Ric$.

Therefore $(SL_n,g)$ is an Einstein manifold (i.e. $Ric$ is a multiple of $g$) with $Ric = - \dfrac{n}{2} g$.  Hence:

$S= -\dfrac{n}{2} dim(SL_n) =  -\dfrac{n}{2} (n^2 -1) = - \dfrac{(n-1)n(n+1)}{2}$.
\end{proof}

\begin{theorem}\label{foliation_GL_SL}
$GL_n = \cup_{c \ne 0} SL_n(c)$ is a foliation of $GL_n$, whose leaves are totally geodesic Semi-Riemannian submanifolds with respect to the metric $g$ of $GL_n$. 

The leaves are Einstein, symmetric, geodesically complete, mutually isometric Semi-Riemannian hypersurfaces with signature $(\dfrac{n(n+1)}{2} -1, \dfrac{n(n-1)}{2})$ and with scalar curvature $S= - \dfrac{(n-1)n(n+1)}{2}$. 

A curve $P=P(t)$ is a geodesic of $(SL_n(c),g)$ if and only if $P(t)= Ke^{tC}$ with $\det (K)=c$, $tr(C)=0$.
\end{theorem}
\begin{proof}
For any $P \in  SL_n(c)$ let us denote again by $g_P$ the restriction to $T_P(SL_n(c))$ of the metric tensor $g$ and let $P_0$ be a fixed point of $SL_n(c)$ (hence $\det (P_0) = c$). Then the left translation $L_{P_0}: (GL_n, g) \to (GL_n, g)$ in an isometry, mapping $SL_n$ onto $SL_n(c)$. Hence the restriction of $L_{P_0}$ is an isometry of $(SL_n, g)$ onto $(SL_n(c), g)$. Moreover, $L_{P_0}$ in an isometry of $(GL_n, g)$, so it transforms totally geodesic Semi-Riemannian submanifolds into totally geodesic Semi-Riemannian submanifolds and this allows to conclude about $SL_n(c)$ . Finally for geodesics we can argue as in Proposition \ref{geodSL}.
\end{proof}

\begin{theorem}\label{GL+_prod}
The manifold $(GL_n^+ ,g)$, with $GL_n^+ = \{ A \in GL_n / \det (A) > 0 \}$, is an open Semi-Riemannian submanifold of $(GL_n ,g)$, isometric to the Semi-Riemannian product manifold $(SL_n \times \mathbb{R}, g \times h)$ where $h=dx^2$ is the euclidean metric on $\mathbb{R}$.
\end{theorem}
\begin{proof}
Note that at any $x \in \mathbb{R}$, $h_x(a,a')=aa'$ for every (tangent vectors) $a, a' \in \mathbb{R}$ and that: 

$(g \times h)_{(P,x)}((V,a), (V',a')) = g_P(V,V') +h_x(a,a') = tr(P^{-1}VP^{-1}V') + aa'$ 

for any $P \in SL_n$, $x \in \mathbb{R}$, $V, V' \in T_P(SL_n)$ (i.e. $tr(P^{-1}V)= tr(P^{-1}V')=0$, $V,V' \in M_n$), $a,a' \in T_x\mathbb{R} = \mathbb{R}$.

We prove that $F: (SL_n \times \mathbb{R}, g \times h) \to (GL_n^+, g)$, defined by $F(P,x) = e^{\frac{x}{\sqrt{n}}}P$, is an isometry. 

Indeed $F$ is of class $C^\infty$ with inverse $F^{-1}: GL_n^+ \to SL_n \times \mathbb{R}$ defined by $F^{-1}(Q)= (\dfrac{Q}{\sqrt[n]{\det (Q)}}, \dfrac{\log (\det (Q))}{\sqrt{n}})$, for any matrix $Q$ with positive determinant (``$\log$'' denotes the natural logarithm). 

We easily get
$$(DF_{(P,x)})(M,a)= e^{\frac{x}{\sqrt{n}}}M+\dfrac{e^{\frac{x}{\sqrt{n}}}}{\sqrt{n}} a P$$ 

for any $P \in SL_n$, $x \in \mathbb{R}$, $a \in T_x(\mathbb{R})$, $M \in T_P(SL_n)$ (i.e. for any $M \in M_n$ such that $tr(P^{-1}M)=0$). 

So, if $tr(P^{-1}M)=tr(P^{-1}M')=0$, we obtain:

\begin{multline*}
g_{_{F(P,x)}}((DF_{(P,x)})(M,a),(DF_{(P,x)})(M',a')) = \\
g_{e^{\frac{x}{\sqrt{n}}}P}(e^{\frac{x}{\sqrt{n}}}(M+\dfrac{a}{\sqrt{n}}P), e^{\frac{x}{\sqrt{n}}}(M'+\dfrac{a'}{\sqrt{n}}P))=
tr((P^{-1}M+\dfrac{aI}{\sqrt{n}})(P^{-1}M'+\dfrac{a'I}{\sqrt{n}})) =\\
 tr(P^{-1}MP^{-1}M') + \dfrac{a}{\sqrt{n}}tr(P^{-1}M') + \dfrac{a'}{\sqrt{n}}tr(P^{-1}M) + \dfrac{aa'}{n}tr(I) = \\
tr(P^{-1}MP^{-1}M') + aa'= g_P(M,M')+h_x(a,a')= (g \times h)_{(P,x)}((M,a),(M',a'))
\end{multline*}

for any $M, M' \in T_P(SL_n)$ and $a,a' \in T_x(\mathbb{R})=\mathbb{R}$.

This means that $F$ is an isometry between $(SL_n \times \mathbb{R}, g \times h)$ and $(GL_n^+, g)$.
\end{proof}

\vspace{0.5cm} \indent {\it
A\,c\,k\,n\,o\,w\,l\,e\,d\,g\,m\,e\,n\,t\,s.\;} We want to thank Giorgio Ottaviani and Fabio Podest\`a for many discussions and their suggestions about the matter of this note.

\bigskip
\begin{center}

\end{center}

\bigskip
\bigskip
\begin{minipage}[t]{10cm}
\begin{flushleft}
\small{
\textsc{Alberto Dolcetti}
\\*Dipartimento di Matematica e Informatica ``Ulisse Dini''
\\*Universit\`a degli Studi di Firenze
\\*Viale Morgagni 67/a
\\*Firenze, 50134, Italia
\\*e-mail: alberto.dolcetti@unifi.it
\\[0.4cm]
\textsc{Donato Pertici}
\\*Dipartimento di Matematica e Informatica ``Ulisse Dini''
\\*Universit\`a degli Studi di Firenze
\\*Viale Morgagni 67/a
\\*Firenze, 50134, Italia
\\*e-mail: donato.pertici@unifi.it
}
\end{flushleft}
\end{minipage}

\end{document}